\documentclass[reqno]{amsart}
\usepackage{hyperref}

\usepackage{amsmath,amssymb,amsthm}
\usepackage{enumerate}
\usepackage{color}
\usepackage{epsfig}
\usepackage{graphics}
\usepackage{graphicx}
\usepackage{subfigure}

\begin{document}
\title[]{on the c*-algebra generated by toeplitz operators and fourier multipliers on the hardy space of a locally compact group}% The title of the article
\dedicatory{Dedicated to Prof. Ayd{\i}n Aytuna on the occasion of
his 65th birthday}
\author[u. g\"{u}l]{u\u{g}ur g\"{u}l}

\address{u\u{g}ur g\"{u}l,  \newline
Hacettepe University, Department of Mathematics, 06800, Beytepe,
Ankara, TURKEY}
\email{\href{mailto:gulugur@gmail.com}{gulugur@gmail.com}}
%\href{mailto:myaddress@wikibooks.org}{myaddress@wikibooks.org}
%\email{{mailto:aozbekler@gmail.com}{aozbekler@gmail.com}}

\thanks{Date: 21/02/2014}

\begin{abstract}
Let $G$ be a locally compact abelian Hausdorff topological group
which is non-compact and whose Pontryagin dual $\Gamma$ is
partially ordered. Let $\Gamma^{+}\subset\Gamma$ be the semigroup
of positive elements in $\Gamma$. The Hardy space $H^{2}(G)$ is
the closed subspace of $L^{2}(G)$ consisting of functions whose
Fourier transforms are supported on $\Gamma^{+}$. In this paper we
consider the C*-algebra $C^{*}(\mathcal{T}(G)\cup
F(C(\dot{\Gamma^{+}})))$ generated by Toeplitz operators with
continuous symbols on $G$ which vanish at infinity and Fourier
multipliers with symbols which are continuous on one point
compactification of $\Gamma^{+}$ on the Hilbert-Hardy space
$H^{2}(G)$. We characterize the character space of this C*-algebra
using a theorem of Power.

\end{abstract}

\maketitle
%\numberwithin{equation}{section}
\newtheorem{theorem}{Theorem}
\newtheorem{acknowledgement}[theorem]{Acknowledgement}
\newtheorem{algorithm}[theorem]{Algorithm}
\newtheorem{axiom}[theorem]{Axiom}
\newtheorem{case}[theorem]{Case}
\newtheorem{claim}[theorem]{Claim}
\newtheorem{conclusion}[theorem]{Conclusion}
\newtheorem{condition}[theorem]{Condition}
\newtheorem{conjecture}[theorem]{Conjecture}
\newtheorem{corollary}[theorem]{Corollary}
\newtheorem{criterion}[theorem]{Criterion}
\newtheorem{definition}[theorem]{Definition}
\newtheorem{example}[theorem]{Example}
\newtheorem{exercise}[theorem]{Exercise}
\newtheorem{lemma}[theorem]{Lemma}
\newtheorem{notation}[theorem]{Notation}
\newtheorem{problem}[theorem]{Problem}
\newtheorem{proposition}[theorem]{Proposition}
\newtheorem{remark}[theorem]{Remark}
\newtheorem{solution}[theorem]{Solution}
\newtheorem{summary}[theorem]{Summary}
\newtheorem*{thma}{Main Theorem}
\newtheorem*{thmc}{Theorem C}
\newtheorem*{thmpower}{Power's Theorem}
\newcommand{\norm}[1]{\left\Vert#1\right\Vert}
\newcommand{\abs}[1]{\left\vert#1\right\vert}
\newcommand{\set}[1]{\left\{#1\right\}}
\newcommand{\Real}{\mathbb R}
\newcommand{\eps}{\varepsilon}
\newcommand{\To}{\longrightarrow}
\newcommand{\BX}{\mathbf{B}(X)}
\newcommand{\A}{\mathcal{A}}

\section*{introduction}
For a locally compact abelian Hausdorff topological group $G$
whose Pontryagin dual $\Gamma$ is partially ordered, one can
define the positive elements of $\Gamma$ as
$\Gamma^{+}=\{\gamma\in\Gamma:\gamma\geq e\}$ where $e$ is the
identity of the group $G$ and the Hardy space $H^{2}(G)$ as
$$H^{2}(G)=\{f\in
L^{2}(G):\hat{f}(\gamma)=0\quad\forall\gamma\not\in\Gamma^{+}\}$$
where $\hat{f}$ is the Fourier transform of $f$. It is not
difficult to see that $H^{2}(G)$ is a closed subspace of
$L^{2}(G)$ and since $L^{2}(G)$ is a Hilbert space there is a
unique orthogonal projection $P:L^{2}(G)\rightarrow H^{2}(G)$ onto
$H^{2}(G)$.

This definition of the Hardy space $H^{2}(G)$ is motivated by
Riesz theorem in the classical cases when $G=\mathbb{T}$ i.e when
$G$ is the unit circle, which characterizes the Hardy class
functions among $f\in L^{2}(\mathbb{T})$ as the space of functions
whose negative Fourier coefficients vanish and by the Paley-Wiener
theorem when $G=\mathbb{R}$, the real line since the group Fourier
transform is the Fourier series when $G=\mathbb{T}$ and coincides
with the Euclidean Fourier transform when $G=\mathbb{R}$.

One can extend the theory of Toeplitz operators to this setting by
defining a Toeplitz operator with symbol $\phi\in L^{\infty}(G)$
as $T_{\phi}=P M_{\phi}$ where $M_{\phi}$ is the multiplication by
$\phi$ and $P$ is the orthogonal projection of $L^{2}$ onto
$H^{2}$. Such a definition was first considered by Coburn and
Douglas in \cite{codou}. However the Toeplitz operators considered
in \cite{codou} were more general since no partial order was
assumed on the dual $\Gamma$ whereas the Hardy space was defined
as the space of functions whose Fourier transforms are supported
on a fixed sub-semigroup $\Gamma_{0}$ of $\Gamma$. The definition
of Hardy space of groups whose duals are partially ordered and
their Toeplitz operators were introduced and studied by Murphy in
\cite{murphy3} and \cite{murphy4}. However in these papers
\cite{murphy3} and \cite{murphy4}, Murphy studies the case where
$G$ is compact. In this paper we will study the case where $G$ is
not compact. One very important assumption that we will make is
that $\Gamma^{+}$ separates the points of $G$, i.e. for any
$t_{1},t_{2}\in G$ satisfying $t_{1}\neq t_{2}$ there is
$\gamma\in\Gamma^{+}$ such that $\gamma(t_{1})\neq\gamma(t_{2})$.

The Toeplitz C*-algebra of a locally compact group is defined as
$$\mathcal{T}(G)=C^{*}(\{T_{\phi}:\phi\in C_{0}(G)\}\cup\{I\})$$
where $C_{0}(G)$ is the space of continuous functions vanishing at
infinity and $I$ is the identity operator. In the study of this
Toeplitz C*-algebra, the most important notions are the commutator
ideal
$com(G)=I^{*}(\{T_{\phi}T_{\psi}-T_{\psi}T_{\phi}:\phi,\psi\in
C_{0}(G)\})$, the semi-commutator ideal
$scom(G)=I^{*}(\{T_{\phi\psi}-T_{\psi}T_{\phi}:\phi,\psi\in
C_{0}(G)\})$ and the symbol map
$\Sigma:C(\dot{G})\rightarrow\mathcal{T}(G)/com(G)$,
$\Sigma(\phi)=[T_{\phi}]$ where $\dot{G}$ is the one point
compactification of $G$ and $[T_{\phi}]$ denotes the equivalence
class of $T_{\phi}$ modulo $com(G)$. It is not difficult to see
that $com(G)\subseteq scom(G)$. We start by proving the following
important result whose proof is adapted from \cite{murphy2}:
\begin{lemma}
Let $G$ be a locally compact abelian Hausdorff topological group
whose Pontryagin dual $\Gamma$ is partially ordered and let
$\Gamma^{+}$ be the semigroup of positive elements of $\Gamma$.
Suppose that $\Gamma^{+}$ separates the points of $G$ i.e. for any
$t_{1},t_{2}\in G$ with $t_{1}\neq t_{2}$ there is
$\gamma\in\Gamma^{+}$ such that $\gamma(t_{1})\neq\gamma(t_{2})$.
Let $com(G)$ and $scom(G)$ be the commutator and the
semi-commutator ideal of the Toeplitz C*-algebra $\mathcal{T}(G)$
respectively. Then
$$com(G)=scom(G)$$
\end{lemma}
It is shown in \cite{codou} and \cite{murphy3} that
$\Sigma:C(\dot{G})\rightarrow\mathcal{T}(G)/com(G)$ is an isometry
but is not a homomorphism since it may not preserve the
multiplication. However
$\Sigma:C(\dot{G})\rightarrow\mathcal{T}(G)/scom(G)$ is a
homomorphism and combining this fact with Lemma 1 above we deduce
that the symbol map
$\Sigma:C(\dot{G})\rightarrow\mathcal{T}(G)/com(G)$ is an
isometric isomorphism which means that
$$M(\mathcal{T}(G))=\dot{G}$$ where $M(A)$ is the character space
of a C*-algebra $A$.

We introduce another class of operators acting on $H^{2}(G)$ which
are called ``Fourier multipliers". These operators in the
classical case $G=\mathbb{R}$ were introduced in \cite{gul}. The
space of Fourier multipliers is defined as $$
F(C(\dot{\Gamma^{+}}))=\{D_{\theta}=\mathcal{F}^{-1}M_{\theta}\mathcal{F}\mid_{H^{2}(G)}:\theta\in
C(\dot{\Gamma^{+}})\}$$ where $\mathcal{F}:L^{2}(G)\rightarrow
L^{2}(\Gamma)$ is the Fourier transform. By Plancherel theorem it
is not difficult to see that the image $\mathcal{F}(H^{2}(G))$ of
$H^{2}(G)$ under the Fourier transform is equal to
$L^{2}(\Gamma^{+})$. Again it is not difficult to see that
$F(C(\dot{\Gamma^{+}}))$ is isometrically isomorphic to
$C(\dot{\Gamma^{+}})$. This means that
$$M(F(C(\dot{\Gamma^{+}})))=\dot{\Gamma^{+}}$$

Lastly we consider the C*-algebra generated by $\mathcal{T}(G)$
and $F(C(\dot{\Gamma^{+}}))$ which we denote by
$\Psi(C_{0}(G),C(\dot{\Gamma^{+}}))$ i.e.
$$\Psi(C_{0}(G),C(\dot{\Gamma^{+}}))=C^{*}(\mathcal{T}(G)\cup
F(C(\dot{\Gamma^{+}})))$$ Using a Theorem of Power
\cite{power},\cite{power2} which characterizes the character space
of the C*-algebra generated by two C*-algebras as a certain subset
of the cartesian product of character spaces of these two
C*-algebras, we prove following theorem:
\begin{thma}
Let $G$ be a non-compact,locally compact abelian Hausdorff
topological group whose Pontryagin dual $\Gamma$ is partially
ordered and let $\Gamma^{+}$ be the semigroup of positive elements
of $\Gamma$. Suppose that $\Gamma^{+}$ separates the points of $G$
i.e. for any $t_{1},t_{2}\in G$ with $t_{1}\neq t_{2}$ there is
$\gamma\in\Gamma^{+}$ such that $\gamma(t_{1})\neq\gamma(t_{2})$.
Let
$$\Psi(C_{0}(G),C(\dot{\Gamma^{+}}))=C^{*}(\mathcal{T}(G)\cup
F(C(\dot{\Gamma^{+}})))$$ be the C*-algebra generated by Toeplitz
operators and Fourier multipliers on $H^{2}(G)$. Then for the
character space $M(\Psi)$ of $\Psi(C_{0}(G),C(\dot{\Gamma^{+}}))$
we have $$M(\Psi)\cong
(\dot{G}\times\{\infty\})\cup(\{\infty\}\times\dot{\Gamma^{+}})$$
\end{thma}
\section{preliminaries}
In this section we fix the notation that we will use throughout
and recall some preliminary facts that will be used in the sequel.

Let $S$ be a compact Hausdorff topological space. The space of all
complex valued continuous functions on $S$ will be denoted by
$C(S)$. For any $f\in C(S)$, $\parallel f\parallel_{\infty}$ will
denote the sup-norm of $f$, i.e. $$\parallel
f\parallel_{\infty}=\sup\{\mid f(s)\mid:s\in S\}.$$ If $S$ is a
locally compact Hausdorff topological space, $C_{0}(S)$ will
denote the space of continuous functions $f$ which vanish at
infinity i.e. for any $\varepsilon>0$ there is a compact subset
$K\subset S$ such that $\mid f(x)\mid<\varepsilon$ for all
$x\not\in K$. For a Banach space $X$, $K(X)$ will denote the space
of all compact operators on $X$ and $\mathcal{B}(X)$ will denote
the space of all bounded linear operators on $X$. The real line
will be denoted by $\mathbb{R}$, the complex plane will be denoted
by $\mathbb{C}$ and the unit circle group will be denoted by
$\mathbb{T}$. The one point compactification of a locally compact
Hausdorff topological space $S$ will be denoted by $\dot{S}$. For
any subset $S\subset$ $\mathcal{B}(H)$, where $H$ is a Hilbert
space, the C*-algebra generated by $S$ will be denoted by
$C^{*}(S)$ and for any subset $S\subset A$ where $A$ is a
C*-algebra, the closed two-sided ideal generated by $S$ will be
denoted by $I^{*}(S)$.

 For any $\phi\in$ $L^{\infty}(G)$ where $G$ is a Borel space(a topological space with a regular measure on it),
$M_{\phi}$ will be the multiplication operator on $L^{2}(G)$
defined as
\begin{equation*}
M_{\phi}(f)(t)=\phi(t)f(t).
\end{equation*}
For convenience, we remind the reader of the rudiments of theory
of Banach algebras, some basic abstract harmonic analysis and
Toeplitz operators.

Let $A$ be a Banach algebra. Then its character space $M(A)$ is
defined as
\begin{equation*}
    M(A)=\{x\in A^{*}:x(ab)=x(a)x(b)\quad\forall a,b\in A\}
\end{equation*}
where $A^{*}$ is the dual space of $A$. If $A$ has identity then
$M(A)$ is a compact Hausdorff topological space with the weak*
topology. When $A$ is commutative $M(A)$ is called the maximal
ideal space of $A$. For a commutative Banach algebra $A$ the
Gelfand transform $\Gamma:A\rightarrow C(M(A))$ is defined as
\begin{equation*}
    \Gamma(a)(x)=x(a).
\end{equation*}
 If $A$ is a commutative C*-algebra with
identity, then $\Gamma$ is an isometric *-isomorphism between $A$
and $C(M(A))$. If $A$ is a C*-algebra and $I$ is a two-sided
closed ideal of $A$, then the quotient algebra $A/I$ is also a
C*-algebra (see \cite{murphy}). For a Banach algebra $A$, we
denote by $com(A)$ the closed ideal in $A$ generated by the
commutators $\{a_{1}a_{2}-a_{2}a_{1}:a_{1},a_{2}\in A\}$. It is an
algebraic fact that the quotient algebra $A/com(A)$ is a
commutative Banach algebra. The reader can find detailed
information about Banach and C*-algebras in \cite{rudin} and
\cite{murphy} related to what we have reviewed so far.

On a locally compact abelian Hausdorff topological group $G$ there
is a unique(up to multiplication by a constant) translation
invariant measure $\lambda$ on $G$ i.e. for any Borel subset
$E\subset G$ and for any $x\in G$, $$\lambda(xE)=\lambda(E)$$
where $xE=\{xy:y\in E\}$ is the translate of $E$ by $x$. This
measure is called the Haar measure of $G$. Let $L^{1}(G)$ be the
space of integrable functions with respect to this measure. Then
$L^{1}(G)$ becomes a commutative Banach algebra with
multiplication as the convolution defined as
$$(f\ast g)(t)=\int_{G}f(ts^{-1})g(s)d\lambda(s)$$
The Pontryagin dual $\Gamma$ of $G$ is defined to be the set of
all continuous homomorphisms from $G$ to the circle group
$\mathbb{T}$:
$$\Gamma=\{\gamma:G\rightarrow\mathbb{T}:\gamma(st)=\gamma(s)\gamma(t)\quad\textrm{and}\quad\gamma\quad\textrm{is
continuous}\}$$ It is a well known fact that $\Gamma$ is in one to
one correspondence with the maximal ideal space $M(L^{1}(G))$ of
$L^{1}(G)$ via the Fourier transform:
$$<\gamma,f>=\hat{f}(\gamma)=\int_{G}\overline{\gamma(t)}f(t)d\lambda(t)$$
When $\Gamma$ is topologized by the weak* topology coming from
$M(L^{1}(G))$, $\Gamma$ becomes a locally compact abelian
Hausdorff topological group with point-wise multiplication as the
group operation:
$$(\gamma_{1}\gamma_{2})(t)=\gamma_{1}(t)\gamma_{2}(t)$$
Let $\tilde{\lambda}$ be a fixed Haar measure on $\Gamma$.
Plancherel theorem asserts that the Fourier transform
$\mathcal{F}$ is an isometric isomorphism of $L^{2}(G)$ onto
$L^{2}(\Gamma)$:
$$\mathcal{F}(f)(\gamma)=\hat{f}(\gamma)=\int_{G}\overline{\gamma(t)}f(t)d\lambda(t)$$
with inverse $\mathcal{F}^{-1}$ defined as
$$\mathcal{F}^{-1}(f)(t)=\check{f}(t)=\int_{\Gamma}\gamma(t)f(\gamma)d\tilde{\lambda}(\gamma)$$
Here we note that $\tilde{\lambda}$ is normalized so that the
above formula for the inverse Fourier transform holds. For
detailed information on abstract harmonic analysis consult
\cite{rudin2}.

A partially ordered group $\Gamma$ is a group with partial order
$\geq$ on it satisfying $\gamma_{1}\geq\gamma_{2}$ implies
$\gamma\gamma_{1}\geq\gamma\gamma_{2}$ $\forall\gamma\in\Gamma$.
This definition of the ordered group was given in \cite{murphy4}.
Let $\Gamma^{+}=\{\gamma\in\Gamma:\gamma\geq e\}$ be the
semi-group of positive elements of $\Gamma$ where $e$ is the unit
of the group $\Gamma$. Let $G$ be a locally compact abelian
Hausdorff topological group and let $\Gamma$ be the Pontryagin
dual of $G$. Then the Hardy space $H^{2}(G)$ is defined as
$$H^{2}(G)=\{f\in
L^{2}(G):\hat{f}(\gamma)=0\quad\forall\gamma\not\in\Gamma^{+}\}$$
The Hardy space $H^{2}(G)$ is a closed subspace of $L^{2}(G)$ and
since $L^{2}(G)$ is a Hilbert space, there is a unique orthogonal
projection $P:L^{2}(G)\rightarrow H^{2}(G)$. For any $\phi\in
L^{\infty}(G)$ the Toeplitz operator $T_{\phi}:H^{2}(G)\rightarrow
H^{2}(G)$ is defined as
$$T_{\phi}=P M_{\phi}$$  Toeplitz operators satisfy the following algebraic properties:
\begin{itemize}
\item $T_{c\phi+\psi}=cT_{\phi}+T_{\psi}$\quad$\forall
c\in\mathbb{C}$,\quad$\forall\phi,\psi\in C(\dot{G})$
\item $T_{\phi}^{*}=T_{\bar{\phi}}$\quad$\forall\phi\in
C(\dot{G})$
\end{itemize} The proofs of these properties are the same as in the classical case where
$G=\mathbb{T}$(or $G=\mathbb{R}$) and can be found in
\cite{douglas}.

The Toeplitz C*-algebra $\mathcal{T}(G)$ is defined to be the
C*-algebra generated by continuous symbols on $G$:
$$\mathcal{T}(G)=C^{*}(\{T_{\phi}:\phi\in C_{0}(G)\}\cup\{I\})$$
where $I$ is the identity operator and $C_{0}(G)$ is the space of
continuous functions which vanish at infinity:
$$C_{0}(G)=\{f:G\rightarrow\mathbb{C}:f\quad\textrm{is continuous
and}\quad\forall\epsilon>0\quad\exists K\subset\subset G\mid
f(t)\mid<\epsilon\quad\forall t\not\in K\}$$ where
$K\subset\subset G$ denotes a compact subset of $G$. Actually one
has $$\mathcal{T}(G)=C^{*}(\{T_{\phi}:\phi\in C(\dot{G})\})$$
where $\dot{G}$ is the one-point compactification of $G$. In the
case where $G$ is compact one has $G=\dot{G}$ and the most
prototypical concrete example of this case is $G=\mathbb{T}$. This
case was analyzed by Coburn in \cite{coburn}. The famous result of
Coburn asserts that for any $T\in\mathcal{T}(\mathbb{T})$ there
are unique $K\in K(H^{2}(\mathbb{T}))$ and $\phi\in C(\mathbb{T})$
such that $T=T_{\phi}+K$. Hence the quotient algebra
$\mathcal{T}(\mathbb{T})/K(H^{2}(\mathbb{T}))$ modulo the compact
operators is isometrically isomorphic to $C(\mathbb{T})$. The two
sided closed *-ideal $com(G)$ generated by the commutators is
called the commutator ideal of $\mathcal{T}(G)$:
$$com(G)=I^{*}(\{T_{\phi}T_{\psi}-T_{\psi}T_{\phi}:\phi,\psi\in
C(\dot{G})\})$$ and the semi-commutator ideal $scom(G)$ is defined
as
$$scom(G)=I^{*}(\{T_{\phi\psi}-T_{\psi}T_{\phi}:\phi,\psi\in
C(\dot{G})\})$$ The symbol map
$\Sigma:C(\dot{G})\rightarrow\mathcal{T}(G)/com(G)$ is defined as
$$\Sigma(\phi)=[T_{\phi}]$$ where $[.]$ denotes the equivalence
class modulo $com(G)$. In \cite{codou} and \cite{murphy4} it is
shown that $\Sigma$ is an isometry. The symbol map $\Sigma$ also
preserves the * operation however is not a homomorphism i.e does
not preserve multiplication. But if $com(G)=scom(G)$ then it is an
isometric isomorphism. We will show under certain conditions that
$com(G)=scom(G)$.

We introduce another class of operators which we call the
``Fourier multipliers". This class of operators in the case
$G=\mathbb{R}$ was introduced in \cite{gul} and proved to be
useful in calculating the essential spectra of a class of
composition operators. The Fourier multiplier
$D_{\theta}:H^{2}(G)\rightarrow H^{2}(G)$ with symbol $\theta\in
C(\dot{\Gamma^{+}})$ is defined as
$$D_{\theta}(f)(t)=(\mathcal{F}^{-1}M_{\theta}\mathcal{F}(f))(t)$$
The most prototypical example of a Fourier multiplier is a
convolution operator with kernel $k\in L^{1}(G)$:
$$(T_{k}f)(t)=\int_{G}k(ts^{-1})f(s)d\lambda(s)$$ It is not
difficult to see that actually $T_{k}=D_{\hat{k}}$ where $\hat{k}$
denotes the Fourier transform of $k$. The set of all Fourier
multipliers $F(C(\dot{\Gamma^{+}}))$ defined as
$$F(C(\dot{\Gamma^{+}}))=\{D_{\theta}:\theta\in
C(\dot{\Gamma^{+}})\}$$ is a commutative C*-algebra since the map
$D:C(\dot{\Gamma^{+}})\rightarrow F(C(\dot{\Gamma^{+}}))$ defined
as $D(\theta)=D_{\theta}$ is an isometric *-isomorphism.

  Lastly we consider the C*-algebra generated by Toeplitz
  operators and Fourier multipliers. Let
  $\Psi(C_{0}(G),C(\dot{\Gamma}))$ be the C*-algebra
  $$\Psi(C_{0}(G),C(\dot{\Gamma}))=C^{*}(\mathcal{T}(G)\cup
  F(C(\dot{\Gamma^{+}})))$$  generated by
  Toeplitz operators with continuous symbols and continuous
  Fourier multipliers. The main result of this paper is a
  characterization of the character space $M(\Psi)$ of
  $\Psi(C_{0}(G),C(\dot{\Gamma}))$. We know that
  $$M(F(C(\dot{\Gamma^{+}})))\cong\dot{\Gamma^{+}},$$ under
  certain conditions we have $scom(G)=com(G)$ and this implies
  that $$M(\mathcal{T}(G))\cong\dot{G}.$$

 We will use the following theorem due to Power
\cite{power},\cite{power2} in identifying the character space of
$\Psi(C_{0}(G),C(\dot{\Gamma^{+}}))$:
\begin{thmpower}
    Let $C_{1}$, $C_{2}$ be C*-subalgebras of $B(H)$ with identity,
    where $H$ is a separable Hilbert space, such that $M(C_{i})\neq$
    $\emptyset$, where $M(C_{i})$ is the space of multiplicative linear
    functionals of $C_{i}$, $i= 1,\,2$ and let $C$ be the C*-algebra that
    they generate. Then for the commutative C*-algebra $\tilde{C}=$
    $C/com(C)$ we have $M(\tilde{C})=$ $P(C_{1},C_{2})\subset$
    $M(C_{1})\times M(C_{2})$, where $P(C_{1},C_{2})$ is defined to be
    the set of points $(x_{1},x_{2})\in$ $M(C_{1})\times M(C_{2})$
    satisfying the condition: \\
    \quad Given $0\leq a_{1} \leq 1$, $0 \leq a_{2} \leq 1$,  $a_{1}\in C_{1}$, $a_{2}\in
    C_{2}$,
\begin{equation*}
    x_{i}(a_{i})=1\quad\textrm{with}\quad
        i=1,2\quad\Rightarrow\quad\| a_{1}a_{2}\|=1.
\end{equation*}
    \label{thmpower}
\end{thmpower}
The proof of this theorem can be found in \cite{power}. Power's
theorem will give the character space $M(\Psi)$ of
$\Psi(C_{0}(G),C(\dot{\Gamma}))$ as a certain subset of the
cartesian product $\dot{G}\times\dot{\Gamma^{+}}$.

\section{the character space of $\Psi(C_{0}(G),C(\dot{\Gamma}))$ }
 In this section we will concentrate on the C*-algebra
 $\Psi(C_{0}(G),C(\dot{\Gamma}))$. But before that we will
 identify the character space $M(\mathcal{T}(G))$ of
 $\mathcal{T}(G)$ under certain conditions. The condition that we
 will pose on $G$ is that $\Gamma^{+}$ separate the points of $G$
 i.e. for any $t_{1},t_{2}\in G$ with $t_{1}\neq t_{2}$ there is
 $\gamma\in\Gamma^{+}$ such that $\gamma(t_{1})\neq\gamma(t_{2})$.
 Under this condition we show that $scom(G)=com(G)$ and this
 implies that $M(\mathcal{T}(G))\cong\dot{G}$. Hence we begin by proving the
 following lemma whose proof is adapted from the proof of Theorem 2.2 of \cite{murphy2}:
 \begin{proposition}
Let $G$ be a locally compact abelian Hausdorff topological group
whose Pontryagin dual $\Gamma$ is partially ordered and let
$\Gamma^{+}$ be the semigroup of positive elements of $\Gamma$.
Suppose that $\Gamma^{+}$ separates the points of $G$ i.e. for any
$t_{1},t_{2}\in G$ with $t_{1}\neq t_{2}$ there is
$\gamma\in\Gamma^{+}$ such that $\gamma(t_{1})\neq\gamma(t_{2})$.
Let $com(G)$ and $scom(G)$ be the commutator and the
semi-commutator ideal of the Toeplitz C*-algebra $\mathcal{T}(G)$
respectively. Then
$$com(G)=scom(G)$$
\end{proposition}

\begin{proof}
It is trivial that $com(G)\subseteq scom(G)$ hence we need to show
that $scom(G)\subseteq com(G)$:

Let $B=\{\phi\in C_{0}(G):T_{\phi}T_{\psi}-T_{\phi\psi}\in
com(G)\}$ then $B$ is a self-adjoint subalgebra of $C_{0}(G)$: Let
$\psi\in B$ then since
$$T_{\phi}T_{\bar{\psi}}-T_{\phi\bar{\psi}}=(T_{\psi}T_{\bar{\phi}}-T_{\bar{\phi}}T_{\psi})^{*}+(T_{\bar{\phi}}T_{\psi}-T_{\bar{\phi}\psi})^{*}$$
we have $(T_{\psi}T_{\bar{\phi}}-T_{\bar{\phi}}T_{\psi})^{*}\in
com(G)$, $(T_{\bar{\phi}}T_{\psi}-T_{\bar{\phi}\psi})^{*}\in
com(G)$ and hence $T_{\phi}T_{\bar{\psi}}-T_{\phi\bar{\psi}}\in
com(G)$ $\forall\phi\in C(\dot{G})$. This implies that
$\bar{\psi}\in B$. It is clear that $\psi_{1},\psi_{2}\in B$
implies that $\psi_{1}+\psi_{2}\in B$. Let us check that
$\psi_{1}\psi_{2}\in B$: we have
$$T_{\phi}T_{\psi_{1}\psi_{2}}-T_{\phi\psi_{1}\psi_{2}}=T_{\phi}(T_{\psi_{1}\psi_{2}}-T_{\psi_{1}}T_{\psi_{2}})+(T_{\phi}T_{\psi_{1}}-T_{\phi\psi_{1}})T_{\psi_{2}}
+(T_{\phi\psi_{1}}T_{\psi_{2}}-T_{\phi\psi_{1}\psi_{2}}).$$ Since
$\psi_{1}\in B$ and $com(G)$ is an ideal we have
$T_{\phi}(T_{\psi_{1}\psi_{2}}-T_{\psi_{1}}T_{\psi_{2}})\in
com(G)$, $(T_{\phi}T_{\psi_{1}}-T_{\phi\psi_{1}})T_{\psi_{2}}\in
com(G)$ and
$(T_{\phi\psi_{1}}T_{\psi_{2}}-T_{\phi\psi_{1}\psi_{2}})\in
com(G)$ which implies that
$T_{\phi}T_{\psi_{1}\psi_{2}}-T_{\phi\psi_{1}\psi_{2}}\in com(G)$
$\forall\phi\in C_{0}(G)$. So we have $\psi_{1}\psi_{2}\in B$. Now
we need to show that $B$ separates the points of $G$ to conclude
the proof since in that case $B$ is closed and by
Stone-Weierstrass theorem we will have $B=C_{0}(G)$: Now let
$A(G)=\{\psi\in C_{0}(G):\psi f\in H^{2}(G)\quad\forall f\in
H^{2}(G)\}$. Clearly $A(G)\subseteq B$, hence if we show that
$A(G)$ separates the points of $G$ we are done. For any $k\in
L^{1}(\Gamma^{+})$ consider
$$\check{k}(t)=\int_{\Gamma^{+}}k(\gamma)\gamma(t)d\tilde{\lambda}(\gamma)$$
then since for any $f\in H^{2}(G)$ we have
$$\mathcal{F}(\check{k}f)=k\ast\hat{f}$$ the Fourier transform of
$\check{k}f$ will be supported in $\Gamma^{+}$. This implies that
$\check{k}\in A(G)$. Since $\Gamma^{+}$ separates the points of
$G$, $\{\check{k}:k\in L^{1}(\Gamma^{+})\}$ also separates the
points of $G$ and hence $A(G)$ separates the points of $G$. This
implies that $B$ separates the points of $G$. This proves our
lemma.
\end{proof}
We have the following corollary of proposition 2:
\begin{corollary}
Let $G$ be a locally compact abelian Hausdorff topological group
whose Pontryagin dual $\Gamma$ is partially ordered and let
$\Gamma^{+}$ be the semigroup of positive elements of $\Gamma$.
Suppose that $\Gamma^{+}$ separates the points of $G$ i.e. for any
$t_{1},t_{2}\in G$ with $t_{1}\neq t_{2}$ there is
$\gamma\in\Gamma^{+}$ such that $\gamma(t_{1})\neq\gamma(t_{2})$.
Let $\mathcal{T}(G)$ be the Toeplitz C*-algebra with symbols in
$C(\dot{G})$ acting on $H^{2}(G)$. Then we have
$$M(\mathcal{T}(G))\cong\dot{G}$$
\end{corollary}
\begin{proof}
The symbol map
$\Sigma:C(\dot{G})\rightarrow\mathcal{T}(G)/com(G)$,
$\Sigma(\phi)=[T_{\phi}]$ is an isometry that preserves the
*-operation and
$\Sigma:C(\dot{G})\rightarrow\mathcal{T}(G)/scom(G)$ is
multiplicative. Since $com(G)=scom(G)$, $\Sigma$ is an isometric
isomorphism. Since characters kill the commutators we have
$$M(\mathcal{T}(G))=M(\mathcal{T}(G)/com(G))\cong
M(C(\dot{G}))=\dot{G}$$
\end{proof}
We will need the following small observation in proving our main
theorem:
\begin{lemma}
Let $G$ be a locally compact,non-compact,abelian Hausdorff
topological group and let $K_{1},K_{2}\subset G$ be two non-empty
compact subsets of $G$. Then there is $t_{0}\in G$ such that
$K_{1}\cap(t_{0}K_{2})=\emptyset$ where $t_{0}K_{2}=\{t_{0}t:t\in
K_{2}\}$.
\end{lemma}
\begin{proof}
Since $G$ is non-compact and locally compact there is a one-point
compactification $\dot{G}$ of $G$. Hence there is a point at
infinity $\infty\in\dot{G}$ such that $\infty\not\in G$. Assume
that the lemma does not hold i.e. there are two non-empty compact
subsets $K_{1},K_{2}\subset G$ such that
$K_{1}\cap(tK_{2})\neq\emptyset$ for all $t\in G$. Now take a net
$\{t_{\alpha}\}_{\alpha\in\mathcal{I}}\subset G$ such that
$\lim_{\alpha\in\mathcal{I}}t_{\alpha}=\infty$. Since
$\forall\alpha\in\mathcal{I}$ we have
$K_{1}\cap(t_{\alpha}K_{2})\neq\emptyset$, there are
$x_{\alpha}\in K_{1}$ and $y_{\alpha}\in K_{2}$ such that
$x_{\alpha}=t_{\alpha}y_{\alpha}$. Since $K_{1}$ and $K_{2}$ are
compact there are $x_{0}\in K_{1}$, $y_{0}\in K_{2}$ and sub-nets
$x_{\alpha_{1}}\in K_{1}$, $y_{\alpha_{2}}\in K_{2}$ such that
$\lim x_{\alpha_{1}}=x_{0}$ and $\lim y_{\alpha_{2}}=y_{0}$. One
can further find a common sub-net index set
$\mathcal{I}_{0}\subset\mathcal{I}$ such that
$\lim_{\beta\in\mathcal{I}_{0}}x_{\beta}=x_{0}$ and
$\lim_{\beta\in\mathcal{I}_{0}}y_{\beta}=y_{0}$. Since
$x_{\beta}=t_{\beta}y_{\beta}$,
$\lim_{\beta\in\mathcal{I}_{0}}t_{\beta}=\infty$ and
multiplication is continuous this implies that
$$x_{0}=\lim_{\beta\in\mathcal{I}_{0}}x_{\beta}=\lim_{\beta\in\mathcal{I}_{0}}t_{\beta}y_{\beta}=\infty$$
but this contradicts to the fact that $x_{0}\in K_{1}$. This
contradiction proves the lemma.
\end{proof}
Now we will show the following lemma which will shorten the proof
of our main theorem. The proof of the following lemma is adapted
from \cite{schmitz}:
\begin{lemma}
Let $G$ be a locally compact abelian Hausdorff topological group
with Pontryagin dual $\Gamma$. Let $\phi\in C_{0}(G)$ and
$\theta\in C_{0}(\Gamma)$ each have compact supports. Then
$D_{\theta}M_{\phi}$ is a compact operator on $L^{2}(G)$ where
$D_{\theta}=\mathcal{F}^{-1}M_{\theta}\mathcal{F}$.
\end{lemma}
\begin{proof}
Let $K_{1}\subset G$ and $K_{2}\subset\Gamma$ be compact supports
of $\phi$ and $\theta$ respectively. Then for any $f\in L^{2}(G)$
we have
\begin{eqnarray*}
&
&(D_{\theta}M_{\phi}f)(t)=\int_{K_{2}}\gamma(t)\theta(\gamma)\left(\int_{K_{1}}\overline{\gamma(\tau)}\phi(\tau)f(\tau)d\lambda(\tau)\right)d\tilde{\lambda}(\gamma)\\
&
&=\int_{K_{1}}\left(\phi(\tau)\int_{K_{2}}\gamma(t\tau^{-1})\theta(\gamma)d\tilde{\lambda}(\gamma)\right)f(\tau)d\lambda(\tau)=\int_{K_{1}}k(t,\tau)f(\tau)d\lambda(\tau)
\end{eqnarray*}
where
$$k(t,\tau)=\phi(\tau)\int_{K_{2}}\gamma(t\tau^{-1})\theta(\gamma)d\tilde{\lambda}(\gamma)$$
Now consider
\begin{eqnarray*}
& &\int_{G}\int_{G}\mid
k(t,\tau)\mid^{2}d\lambda(t)d\lambda(\tau)=\int_{G}\int_{G}\mid\phi(\tau)\int_{K_{2}}\gamma(t\tau^{-1})\theta(\gamma)d\tilde{\lambda}(\gamma)\mid^{2}d\lambda(t)d\lambda(\tau)\\
&
&\leq\parallel\phi\parallel_{\infty}^{2}\int_{K_{1}}\int_{G}\mid\check{\theta}(t\tau^{-1})\mid^{2}d\lambda(t)d\lambda(\tau)=\parallel\phi\parallel_{\infty}^{2}\int_{K_{1}}\parallel\check{\theta}\parallel_{2}d\lambda(\tau)\\
&
&=\parallel\phi\parallel_{\infty}^{2}\int_{K_{1}}\parallel\theta\parallel_{2}d\lambda(\tau)=\parallel\phi\parallel_{\infty}^{2}\parallel\theta\parallel_{2}\lambda(K_{1})<\infty
\end{eqnarray*}
This implies that $D_{\theta}M_{\phi}$ is Hilbert-Schmidt and
hence compact.
\end{proof}
 Now we are ready to prove our main theorem as follows:
\begin{thma}
Let $G$ be a non-compact,locally compact abelian Hausdorff
topological group whose Pontryagin dual $\Gamma$ is partially
ordered and let $\Gamma^{+}$ be the semigroup of positive elements
of $\Gamma$. Suppose that $\Gamma^{+}$ separates the points of $G$
i.e. for any $t_{1},t_{2}\in G$ with $t_{1}\neq t_{2}$ there is
$\gamma\in\Gamma^{+}$ such that $\gamma(t_{1})\neq\gamma(t_{2})$.
Let
$$\Psi(C_{0}(G),C(\dot{\Gamma^{+}}))=C^{*}(\mathcal{T}(G)\cup
F(C(\dot{\Gamma^{+}})))$$ be the C*-algebra generated by Toeplitz
operators and Fourier multipliers on $H^{2}(G)$. Then for the
character space $M(\Psi)$ of $\Psi(C_{0}(G),C(\dot{\Gamma^{+}}))$
we have $$M(\Psi)\cong
(\dot{G}\times\{\infty\})\cup(\{\infty\}\times\dot{\Gamma^{+}})$$
\end{thma}
\begin{proof}
We will use Power's Theorem. In the setup of Power's theorem
$C_{1}=\mathcal{T}(G)$ and $C_{2}=F(C(\dot{\Gamma^{+}}))$. By
corollary 3 we have $M(C_{1})=\dot{G}$ and we have
$M(C_{2})=\dot{\Gamma^{+}}$. So we need to determine
$(t,\gamma)\in\dot{G}\times\dot{\Gamma^{+}}$ satisfying for
$0\leq\phi,\theta\leq 1$, $\phi(t)=\theta(\gamma)=1$ implies
$\parallel T_{\phi}D_{\theta}\parallel=1$.

Let $(t,\gamma)\in G\times\Gamma^{+}$. Let $\phi\in C(\dot{G})$,
$\theta\in C(\dot{\Gamma^{+}})$ such that $0\leq\phi,\theta\leq 1$
and $\phi(t)=\theta(\gamma)=1$. Let us also assume that $\theta$
and $\phi$ have compact supports. Let $\tilde{\theta}\in
C(\dot{\Gamma})$ such that $0\leq\tilde{\theta}\leq 1$,
$\tilde{\theta}$ has compact support and
$\tilde{\theta}\mid_{C(\dot{\Gamma^{+}})}=\theta$. Since
$$\parallel T_{\phi}D_{\theta}\parallel\leq\parallel
M_{\phi}D_{\tilde{\theta}}\parallel_{L^{2}(G)}$$ it suffices to
show that $\parallel
M_{\phi}D_{\tilde{\theta}}\parallel_{L^{2}(G)}<1$. We will also
assume that $\phi(s)<1$ $\forall s\in G-\{t\}$. Since
$(M_{\phi}D_{\tilde{\theta}})^{*}=D_{\tilde{\theta}}M_{\phi}$ and
$D_{\tilde{\theta}}M_{\phi}$ is compact by Lemma 5,
$M_{\phi}D_{\tilde{\theta}}$ is also compact. Hence
$M_{\phi}D_{\tilde{\theta}}(M_{\phi}D_{\tilde{\theta}})^{*}=M_{\phi}D_{\tilde{\theta}}^{2}M_{\phi}$
is a compact self-adjoint operator on $L^{2}(G)$ and this implies
that $\parallel
M_{\phi}D_{\tilde{\theta}}^{2}M_{\phi}\parallel=\mu$ where $\mu$
is the largest eigenvalue of
$M_{\phi}D_{\tilde{\theta}}^{2}M_{\phi}$. Let $f\in L^{2}(G)$ be
the corresponding eigenvector such that $\parallel
f\parallel_{2}=1$, then we have
$$\mu=\parallel\mu
f\parallel_{2}=\parallel(M_{\phi}D_{\tilde{\theta}}^{2}M_{\phi}f)\parallel_{2}<\parallel
D_{\tilde{\theta}}^{2}M_{\phi}f\parallel_{2}\leq 1$$ since
$\phi(s)<1$ $\forall s\in G-\{t\}$. This implies that $\parallel
D_{\tilde{\theta}}M_{\phi}\parallel^{2}=\parallel
M_{\phi}D_{\tilde{\theta}}^{2}M_{\phi}\parallel <1$. This means
that $(t,\gamma)\not\in M(\Psi)$ $\forall(t,\gamma)\in
G\times\Gamma^{+}$. So if $(t,\gamma)\in M(\Psi)$ then either
$t=\infty$ or $\gamma=\infty$.

Now let $t\in\dot{G}$ and $\gamma=\infty$. Let $\phi\in
C(\dot{G})$ and $\theta\in C(\dot{\Gamma^{+}})$ such that
$0\leq\phi,\theta\leq 1$ and $\phi(t)=\theta(\infty)=1$. Observe
that $P=D_{\chi_{\Gamma^{+}}}$ where $\chi_{\Gamma^{+}}$ is the
characteristic function of $\Gamma^{+}$. So we have
$D_{\theta}T_{\phi}=D_{\theta}D_{\chi_{\Gamma^{+}}}M_{\phi}=D_{\chi_{\Gamma^{+}}}D_{\theta}M_{\phi}=D_{\theta}M_{\phi}$.
Since $\mathcal{F}$ is unitary we have $$\parallel
D_{\theta}M_{\phi}\parallel_{H^{2}(G)}=\parallel\mathcal{F}D_{\theta}M_{\phi}\mathcal{F}^{-1}\parallel_{L^{2}(\Gamma^{+})}=\parallel
M_{\theta}\mathcal{F}M_{\phi}\mathcal{F}^{-1}\parallel_{L^{2}(\Gamma^{+})}$$
Since $\theta(\infty)=1$ we have $\forall\epsilon>0$,
$\exists\gamma_{0}\in\Gamma^{+}$ such that
$1-\epsilon\leq\theta(\gamma)\leq 1$
$\forall\gamma\geq\gamma_{0}$. Consider the operator
$S_{\gamma_{0}}:L^{2}(\Gamma^{+})\rightarrow L^{2}(\Gamma^{+})$
defined as $(S_{\gamma_{0}}f)(\gamma)=f(\gamma_{0}^{-1}\gamma)$,
then $S_{\gamma_{0}}$ is an isometry. Observe that
\begin{eqnarray*}
& &(\mathcal{F}^{-1}S_{\gamma_{0}}f)(t)=(\check{S_{\gamma_{0}}f})(t)=\int_{\Gamma^{+}}\gamma(t)f(\gamma_{0}^{-1}\gamma)d\tilde{\lambda}(\gamma)\\
& &=
\int_{\Gamma^{+}}\gamma_{0}(t)u(t)f(u)d\tilde{\lambda}(u)=\gamma_{0}(t)\check{f}(t)=(M_{\gamma_{0}}\check{f})(t)
\end{eqnarray*}
Hence we have
$S_{\gamma_{0}}=\mathcal{F}M_{\gamma_{0}}\mathcal{F}^{-1}$ which
implies that
$$S_{\gamma_{0}}(\mathcal{F}M_{\phi}\mathcal{F}^{-1})=(\mathcal{F}M_{\phi}\mathcal{F}^{-1})S_{\gamma_{0}}.$$
Now let $f\in L^{2}(\Gamma^{+})$ such that
$\parallel(\mathcal{F}M_{\phi}\mathcal{F}^{-1})f\parallel_{2}>1-\epsilon$
and $\parallel f\parallel_{2}=1$ then for
$g=\mathcal{F}M_{\phi}\mathcal{F}^{-1}f$ we have $$\parallel
M_{\theta}S_{\gamma_{0}}g\parallel_{2}\geq(1-\epsilon)^{2}$$ since
$S_{\gamma_{0}}g$ is supported on
$\{\gamma:\gamma\geq\gamma_{0}\}$, $\theta(\gamma)\geq 1-\epsilon$
$\forall\gamma\geq\gamma_{0}$ and $\parallel
S_{\gamma_{0}}g\parallel_{2}\geq 1-\epsilon$. Since
$S_{\gamma_{0}}g=(\mathcal{F}M_{\phi}\mathcal{F}^{-1})S_{\gamma_{0}}f$
we have
$$\parallel(M_{\theta}\mathcal{F}M_{\gamma_{0}}\mathcal{F}^{-1})(S_{\gamma_{0}}f)\parallel_{2}\geq(1-\epsilon)^{2}.$$
Since $S_{\gamma_{0}}$ is an isometry we have $\parallel
S_{\gamma_{0}}f\parallel_{2}=1$ and this implies that $$\parallel
M_{\theta}\mathcal{F}M_{\phi}\mathcal{F}^{-1}\parallel\geq(1-\epsilon)^{2}$$
$\forall\epsilon>0$. Therefore we have $\parallel
M_{\theta}\mathcal{F}M_{\phi}\mathcal{F}^{-1}\parallel=\parallel
D_{\theta}T_{\phi}\parallel=1$. Hence $(t,\infty)\in M(\Psi)$
$\forall t\in\dot{G}$.

Now let $\gamma\in\dot{\Gamma^{+}}$ and $t=\infty$. Let $\phi\in
C(\dot{G})$ and $\theta\in C(\dot{\Gamma^{+}})$ such that
$0\leq\phi,\theta\leq 1$ and $\phi(\infty)=\theta(\gamma)=1$.
Since $\phi(\infty)=1$,for any $\epsilon>0$ there is a compact
subset $K_{1}\subset G$ such that $1-\epsilon\leq\phi(t)\leq 1$
$\forall t\not\in K_{1}$. Let
$\tilde{\theta}=\chi_{\Gamma^{+}}\theta$. Then we have
$$D_{\theta}T_{\phi}=D_{\theta}D_{\chi_{\Gamma^{+}}}M_{\phi}=D_{\chi_{\Gamma^{+}}\theta}M_{\phi}=D_{\tilde{\theta}}M_{\phi}.$$
Let $\epsilon>0$ be given. Let $g\in H^{2}(G)$ so that $\parallel
g\parallel_{2}=1$ and $\parallel
D_{\tilde{\theta}}g\parallel_{2}\geq 1-\epsilon$. Let
$K_{2}\subset G$ be a compact subset of $G$ so that
$$(\int_{K_{2}}\mid g(t)\mid^{2}d\lambda(t))^{\frac{1}{2}}\geq
1-\epsilon.$$ By Lemma 4 we have $t_{0}\in G$ such that
$K_{1}\cap(t_{0}K_{2})=\emptyset$. Let
$(S_{t_{0}}g)(t)=g(tt_{0}^{-1})$ then $(\int_{t_{0}K_{2}}\mid
S_{t_{0}}g(t)\mid^{2}d\lambda(t))^{\frac{1}{2}}=(\int_{K_{2}}\mid
g(t)\mid^{2}d\lambda(t))^{\frac{1}{2}}\geq 1-\epsilon.$ and this
implies that $$\parallel
S_{t_{0}}g-M_{\phi}S_{t_{0}}g\parallel_{2}\leq 2\epsilon.$$ We
observe that
$S_{t_{0}}=\mathcal{F}^{-1}M_{\hat{t_{0}}}\mathcal{F}$ where
$\hat{t_{0}}:\Gamma\rightarrow\mathbb{C}$ defined as
$\hat{t_{0}}(\gamma)=\gamma(t_{0})$. This implies that $S_{t_{0}}$
is unitary and we have
$D_{\tilde{\theta}}S_{t_{0}}=S_{t_{0}}D_{\tilde{\theta}}$. Since
$\parallel D_{\tilde{\theta}}\parallel=1$ we have $$\parallel
D_{\tilde{\theta}}S_{t_{0}}g-D_{\tilde{\theta}}M_{\phi}S_{t_{0}}g\parallel_{2}\leq
2\epsilon.$$ Since $S_{t_{0}}$ is unitary for $f=S_{t_{0}}g$ we
have $\parallel f\parallel_{2}=1$ and
$D_{\tilde{\theta}}S_{t_{0}}=S_{t_{0}}D_{\tilde{\theta}}$ together
with $\parallel D_{\tilde{\theta}}g\parallel_{2}\geq 1-\epsilon$
implies that $$\parallel
D_{\tilde{\theta}}M_{\phi}f\parallel_{2}\geq 1-3\epsilon.$$ Since
$\epsilon>0$ is arbitrary we have $\parallel
D_{\tilde{\theta}}M_{\phi}\parallel=\parallel
D_{\theta}T_{\phi}\parallel=1$. Therefore we have
$(\infty,\gamma)\in M(\Psi)$, $\forall\gamma\in\dot{\Gamma^{+}}$.
Our theorem is thus proven.
\end{proof}
\section{Acknowledgements}
The author wishes to express his sincere thanks to Prof. R{\i}za
Ert\"{u}rk of Hacettepe University for useful discussions on Lemma
4.
% ----------------------------------------------------------------
\bibliographystyle{amsplain}
%\bibliography{xbib}
\addtocontents{toc}{\protect\contentsline {part}{}{}}

\end{document}